\documentclass{article}
\usepackage{amsthm, amsfonts, amssymb,latexsym, amsmath,hyperref}

\title{An Energy Bound in the Affine Group}

\usepackage{tikz}

\usepackage{amsmath, comment}
\usepackage[a4paper, total={6in, 8in}]{geometry}
\usepackage{fixme}
\fxsetup{
    status=draft,
    author=,
    layout=margin,
    theme=color
}

\usepackage{color, enumerate}

\author{Giorgis Petridis, Oliver Roche-Newton, Misha Rudnev and Audie Warren}

\newtheorem{theorem}{Theorem}[section]

\newtheorem{corollary}[theorem]{Corollary}
\newtheorem{lemma}[theorem]{Lemma}
\begin{document}
 \maketitle

\begin{abstract} 
We prove a nontrivial energy bound for a finite set of affine transformations over a general field and discuss a number of implications. These include new bounds on growth in the affine group, a quantitative version of a theorem by Elekes about rich lines in grids. We also give a positive answer to a question of Yufei Zhao that for a plane point set $P$ for which no line contains a positive proportion of points from $P$, there may be at most one line, meeting the set of lines defined by $P$ in at most a constant multiple of $|P|$ points.
\end{abstract}

 \section{Introduction} 
 
 Let $\mathbb F$ be a field with multiplicative group $\mathbb{F}^*$. We are particularly interested in the zero characteristic case $\mathbb F=\mathbb R$ (or $\mathbb C$, there is no difference as far as the results we present are concerned) or $\mathbb F =\mathbb{F}_p$, where $p$ is an odd prime (there is no difference as to our results between $\mathbb{F}_p$ and a general $\mathbb{F}$ with characteristic $p$).

Let $A$ be a finite set of transformations in the affine group $G=\emph{Aff}(\mathbb{F}) = \mathbb{F}\rtimes \mathbb{F}^*$. For $g\in G$ we write $g=(g_1,g_2)\in \mathbb{F}^*\times \mathbb{F}$ and visualise it as a point in $\mathbb{F}^2$ with the $y$-axis deleted. $G$ has a unipotent normal subgroup $U=\{(1,x):x\in \mathbb{F}\}$ and maximal tori $T=T(x)=\emph{Stab}(x)$, for the action $g(x)=g_1x+g_2$ on $\mathbb{F}$. The tori correspond to finite slope lines through the identity $(1,0)$ in the coordinate plane. Their left cosets are given by families of parallel lines with the given slope; their right cosets by families of lines incident to the deleted $y$-axis at the point where the line corresponding to $T$ itself meets it. Cosets of $U$ are vertical lines.

Below we use the notation $A^k$ for the $k$-fold Cartesian product of $A$ with itself, not to be confused with $A^{-1}$, the set of inverses of elements of $A$ or the product sets $A^{-1}A, AA$, etc. We define two versions of the energy of $A$ as
\begin{equation}\label{e:eng}
E(A) := | \{(g,h,u,v)\in A^4:\, g^{-1}\circ h= u^{-1}\circ v\}|\,,\; E^*(A)  :=|\{(g,h,u,v)\in A^4:\, g\circ h= u \circ v\}|\,.
\end{equation}
These are different quantities. A standard application of the Cauchy-Schwarz inequality yields the following bounds connecting $E(A)$ and $E^*(A)$ with $A^{-1}A$ and $AA$, respectively.
\begin{equation} \label{CSbasic}
E(A) \geq \frac{|A|^4}{|A^{-1}A|}\,,\,\,\,\,\,\,\,\,\,E^*(A) \geq \frac{|A|^4}{|AA|}\,.
\end{equation}
In fact, Shkredov \cite[Section 4]{Sh} proved the inequality $E^*(A)\leq E(A)$.

Also, for a scalar set $S$ we use the standard notations $E^+(S)$ (and respectively $E^\times(S)$) for its additive (and respectively multiplicative) energy.

Note that in principle one cannot expect to have a nontrivial upper bound on $E(A)$, since $A$ can lie in a coset of an Abelian subgroup. But  the paradigm of the rich theory of growth in groups, developed in the past 15 years, is that this is the only case when $A$ would not grow by multiplication. The methods it uses, however, are not only very general, applying to a wide class of finite non-commutative groups to yield ``reasonable''  quantitative estimates on the growth rate, but also technically work with sets, rather than their energies. Thus but for the case of sufficiently large sets, where one can use Fourier analysis and representation theory to describe {\em quasirandom group actions} (see  e.g. \cite{Gill} and references therein), the only way to connect this with the energy of a set of transformations has been via the non-commutative Balog-Szemer\'edi-Gowers Theorem.

The first nontrivial instance of a non-commutative group is the affine group. An element $(g_1,g_2) \in G$ can be identified with the line $y=g_1x+g_2$. In \cite{RuSh} the following incidence theorem was proved, establishing an explicit connection between energy in the affine group and incidence theory.\footnote{Here and throughout the paper, for positive values $X$ and $Y$, the notation $X \gg Y$ and $Y \ll X$ indicates that there is an absolute constant $C>0$ such that $X \geq CY$.}

\begin{theorem}	\label{t:Elekes_new}
	Let $S, T\subseteq \mathbb{F}$ be finite scalar sets and ${A}$  a finite set of non-vertical, non-horizontal lines in $\mathbb{F}^2$. \\
	If $\mathbb{F} = \mathbb{R}$, then
	\begin{equation}\label{f:Elekes_new_R}
	{I} (S\times T, {A}) \ll |T|^{1/2} |S|^{2/3} E^{1/6}({A}) |{A}|^{1/3} +  |T|^{1/2} |{A}| \,.
	\end{equation}
	If  $\mathbb{F} = \mathbb{F}_p$, then 
	\begin{equation}\label{f:Elekes_new_Fp}
	{I} (S\times T, {A}) \ll |T|^{1/2} |S|^{5/8} E^{1/8}({A}) |{A}|^{1/2} + |T|^{1/2} |{A}| \cdot  \sqrt{\max\{ 1, |S|^2/p \}} \,.
	\end{equation}
\end{theorem}
Theorem \ref{t:Elekes_new} (along with other results) enabled \cite[Theorem 7]{RuSh}, which is a somewhat stronger version of the following classical theorem of Elekes \cite{Elekes1} on {\em rich lines in grids}.

\begin{theorem}[Elekes]
	Let $\alpha \in (0,1)$, $n$ be a positive integer, $A \subset \mathbb R$ with $|A|=n$ and suppose there are $n$ non-horizontal and non-vertical lines intersecting $A\times A$ in at least $\alpha n$ points.
	Then either\\
	$\bullet~$ at least $\gg \alpha^C n$ of these lines are parallel, or\\
	$\bullet~$ at least $\gg \alpha^C n$ of these lines are incident to a common point.\\
	Here $C>0$ is an absolute constant.
	\label{t:Elekes}
\end{theorem}
As we have indicated above, parallel and concurrent lines are in correspondence with coset families in $\emph{Aff} (\mathbb{R})$ and hence Theorem \ref{t:Elekes} is a result on affine transformations. Theorem \ref{t:Elekes} and underlying ideas were further developed by several authors, having inspired Murphy's work \cite{Brendan_rich}, which deals with $\emph{Aff} (\mathbb{F})$, $\mathbb{F}$ being $\mathbb C$ or finite. Murphy's paper, as well as \cite{RuSh}, used and combined ideas of Elekes with Helfgott's theory of  growth in groups, exposed in \cite{HaH}.

Theorem \ref{t:Elekes} is {\em qualitative} in the sense that its proofs so far failed to yield a reasonable quantitative value of $C$, for the proofs involved a graph-theoretical Balog-Szemer\'edi-Gowers type argument, which is rather costly for quantitative estimates. 

However, it would not be difficult to derive Theorem \ref{t:Elekes} as a corollary of Theorem \ref{t:Elekes_new}, provided that there is a nontrivial bound on the energy $E(A)$ therein. Such estimates were obtained in \cite{RuSh} in the special case of $A=C\times D$ being a Cartesian product. The  bound obtained in \cite{RuSh} was that for some $\delta>0$, 
\begin{equation} \label{gridthreshold}
E(C \times D) \ll |C|^{5/2-\delta}|D|^3,
\end{equation}
provided that the sizes of $C$ and $D$ are not vastly different. The quantitative lower bound on $\delta$ would be, if attempted, quite small. On the other hand, the above estimate with $\delta=0$ can be regarded in a sense as a threshold one, and having any $\delta>0$  makes it strong enough to reach the point at which new information in the {\em arithmetic}  setting of sum-product theory is obtained, where incidence configurations involving such Cartesian products frequently arise. Several such applications were explored in \cite{RNW}, and also in \cite{RuSh}.

\section{Main results}
 
In this paper we prove the generalisation of the threshold version of estimate \eqref{gridthreshold} for any finite set $A\subset\emph{Aff}(\mathbb{F})$. This enables us to develop a number of {\em geometric} combinatorics estimates, which so far have been inaccessible. Our main result is the following:

\begin{theorem} \label{thm:main}

Let $A$ be a finite set of transformations in the affine group $\emph{Aff}(\mathbb{F})$ such that no line contains more than $M$ points of $A$, and no vertical line contains more than $m$ points of $A$. Also, suppose in positive characteristic $p,$ one has  $m|A|\leq p^2$. Then,
\[
\max\{ E(A) , E^*(A) \}  \ll  m^{\frac{1}{2}}|A|^{\frac{5}{2}} +  M|A|^2 .
\]
\end{theorem}
Note that if $A$ is taken to be a Cartesian product, this recovers \eqref{gridthreshold}, with $\delta=0$.

We remark that one may consider the asymmetric case of $E(A,B)$ for $A,B\subset \emph{Aff}(\mathbb{F})$, defined similar to $E(A)$ in \eqref{e:eng}, only now $g,u\in A;\,h,v\in B$.  Give the notations $m,M$ of Theorem \ref{thm:main} subscripts ${}_{A\,\text{or}\,B}$, depending on whether they pertain to $A$ or $B$. The estimate of Theorem \ref{thm:main} then changes to
$$
E(A,B)\ll  m^{\frac{1}{2}}_A |A||B|^{\frac{3}{2}} + \min\{M_A,M_B\}|A||B|\,, 
$$
provided that $m_A|B|\leq p^2$ in positive characteristic. 
Moreover, if $\mathbb{F}=\mathbb{F}_p$ one can dispense with the $p$-constraint, by using the asymptotic version of Theorem \ref{thm:pointplane}, estimate \eqref{pp_ass}. This would add to the right-hand side of the latter estimate a (crudely estimated) term $\frac{m_A|A||B|^2}{p}.$

 What is the correct bound for $E(A)$? The example $A=C\times D$, where $C$ is a geometric progression and $D$ an arithmetic one, with $|D|>|C|$ shows, in view of the forthcoming equations \eqref{eneqs}, that in this case $E(A)\gg M|A|^2$. Note that owing to energy quadruples $(g,h,u,v)$ lying in the same (left) coset of some torus, the upper bound on $E(A)$ should exceed the number of collinear point triples in $A$. We carefully conjecture that $M|A|^2$ may be the correct asymptotics, possibly with additional subpolynomial factors in $|A|$, although we do not have evidence for their presence. Something similar has been conjectured for the concept of bisector energy in \cite{LSdZ}. There are other parallels between bisector energy and our current considerations; see the appendix for more discussion. 

However, the key Lemma \ref{lem:Cfixed} used in the proof of Theorem \ref{thm:main} is itself sharp, and so is  estimate \eqref{triv} that combined with the latter lemma finishes the proof of Theorem \ref{thm:main}. See the forthcoming remark after the proof of Theorem~\ref{thm:main}.

The proof of Theorem \ref{thm:main} is geometric, relying on incidence theory. After an application of Cauchy-Schwarz in the form of \eqref{CSbasic}, it enables an estimate on growth in $\emph{Aff}(\mathbb{F}_p)$ which is much stronger than its algebra-based predecessors in \cite[Proposition 4.8]{HaH}, \cite[Theorem 27]{Brendan_rich} and even the geometric \cite[Theorem 5]{RuSh} which took advantage of the  (sharp) estimate on the number of directions, determined by a point set in $\mathbb{F}_p^2$ by Sz\H onyi \cite{Sz}.\footnote{The conditions of Theorem \ref{thm:main} and Corollary \ref{c:main} are identical. They may appear different as one statement uses geometric and another uses algebraic language.}

\begin{corollary} \label{c:main}
Let $A\subset \emph{Aff}(\mathbb{F}_p)$ have no more than $M$ elements in a coset of a torus and no more than $m$ elements in a coset of $U$. Suppose $m|A|\leq p^2$. Then
\[
\min\{|AA|,|A^{-1}A|\} \gg  m^{-1/2}|A|^{3/2} +  M^{-1}|A|^2\,.
\]
In particular, if $|AA|=K|A|,$ then $A$ has $\gg |A|/K^2$ elements in a coset of the unipotent subgroup or $\gg |A|/K$ elements in a coset of a torus.
\end{corollary}

Theorem \ref{thm:main} enables us to establish the following quantitative version of Theorem \ref{t:Elekes} over general fields.

\begin{theorem} 
Let $n,k$ be positive integers, $A \subset \mathbb F$ with $|A|=n$ and suppose there are $k$ non-horizontal and non-vertical lines intersecting $A\times A$ in at least $\alpha n$ points for some $n^{-\frac{1}{2}}\ll \alpha\leq 1$. In positive characteristic $p$ also assume $$
p\gg \max\{k,\alpha^{-2}n\}\,.$$
Then for $C=12$ if $\mathbb{F}=\mathbb{R}$ and $C=16$ if 
	$\mathbb{F}=\mathbb{F}_p$, we have either\\
	$\bullet~$ at least $\gg \alpha^{C} n^{-2}k^{3}$ of these lines are parallel, and moreover $E^{+} (A) \gg \alpha^{2+C} k^3$; or  \\
	$\bullet~$ at least $\gg \alpha^{C/2} n^{-1}k^2$ of these lines are incident to a common point, and
	moreover there is $s\in \mathbb{F}$ such that  $E^{\times} (A-s) \gg \alpha^{2+\frac{C}{2}}nk^2$.
   
	\label{t1:Elekes}\end{theorem}
	
\subsection*{Two applications in line geometry}

\subsubsection*{Shadow of a point set on two lines} Theorem \ref{t1:Elekes}  gives an affirmative answer (with a quantitative estimate) to a conjecture stated in \cite[Conjecture 8]{RNW} -- a question that had earlier been raised by Yufei Zhao. As a matter of fact, it follows from the proof below that the positive qualitative answer was implicit by the Elekes' Theorem \ref{t:Elekes}. We give a statement over the reals, although it applies to a general $\mathbb{F}$, given that $P$ defines $\gg|P|^2$ lines. 

Given a planar point set $P$ and a line $l$ not incident to $P$, we call the {\em shadow} of $P$ on $l$ the set of points where $l$ meets the lines from $L(P)$ -- the set of lines, generated by pairs of points in $P$. Abusing notation slightly, we write $ L(P) \cap l$ for the shadow of $P$ on $l$.

\begin{corollary} \label{thm:lineintersection}
Let $P \subset \mathbb{R}\mathbb{P}^2$ be a finite, sufficiently large set of points, and let $l_1$ and $l_2$ be two distinct lines such that $\ll |P|^{1-\delta_1}$ points of $P$ lie on any line passing through the intersection of $l_1$ and $l_2$, and $\ll|P|^{1-\delta_2}$ points of $P$ lie on any other line, for some $\delta_1, \delta_2 > 0$. Then 
$$\max\{ |L(P) \cap l_1| , |L(P) \cap l_2| \} \gg  \min\{|P|^{1 + \frac{\delta_1}{14}},|P|^{1 + \frac{\delta_2}{7}} \}.$$
\end{corollary}
Roughly, a sufficiently non-collinear point set $P$ can have a small shadow on at most one line.

This question is partially motivated by the sum-product problem. Indeed, if we take the point set to be $P=\{(a,a^2) : a \in A\subset \mathbb{R}\}$, $l_1$ to be the line at infinity and $l_2$ to be the $y$-axis, Corollary \ref{thm:lineintersection}, used as a black box implies that
\[
|A+A| + |AA| \gg |A|^{1+c}\,,
\]
with $c=\frac{1}{14}$. 
This implication is true for sufficiently small sets in positive characteristic as well, for the role of $\mathbb{R}$ in the proof of Corollary \ref{thm:lineintersection} is merely to guarantee that $|L(P)|\gg |P|^2$. But in fact, if $P$ is a parabola as considered above, it's easy to show that (as a set of affine transformations) $E(P)\ll|P|^2$, and following the argument in the proof of Corollary \ref{thm:lineintersection}, one gets $c=\frac{1}{7}$ in characteristic zero and $c=\frac{1}{9}$ in positive characteristic.

A special case of Corollary \ref{thm:lineintersection} was proven in \cite{RNW} and used to prove that
\[
\left| \left \{ \frac{ab-cd}{a-c} : a,b,c,d \in A \right \} \right | \gg |A|^{2+\frac{1}{14}},
\]
for any finite $A \subset \mathbb R$. By making suitable choices for the point set and the two fixed lines, one can also recover quantitatively weaker versions of several well-known sum-product results, such as lower bounds on $\max \{|A+A|,|A^{-1}+A^{-1}|\}$ and $\max \{|AA|, |(A+1)(A+1)|\}$.

We use Corollary \ref{thm:lineintersection} to point out another curious fact, concerning the shadow of the Cartesian product $A \times A$ on a single line.
\begin{corollary}\label{cor:cartesianintersection}
Let $A \subset \mathbb{R}$ be a finite set, and $P = A \times A$. Let $l$ be any affine line not of the form $y=x$ or $y = -x + k$ for some $k$. Then we have
$$|L(A \times A) \cap l| \gg |A|^{2 + \frac{1}{14}}.$$
\end{corollary}

\subsubsection*{Quadrangles, rooted on two lines} The other application is as follows. Consider a planar point set $P$, not incident to the $y$-axis. What is the maximum number of non-trivial quadrangles $(g,h,u,v)$ with vertices in $P$, so that the two lines connecting $g,h$ and $u,v$ are parallel, i.e., meet the line at infinity at the same point, and the two lines, connecting $g,u$ and $h,v$ have the same $y$-intercept? By non-trivial, we mean here that the four points $g,h,u,v$ do not lie on the same line. The question and the forthcoming answer applies to any two lines  $l_1$, $l_2$, not incident to the point set $P$. Let $\mathcal Q(P)$ be the set of quadruples in question.


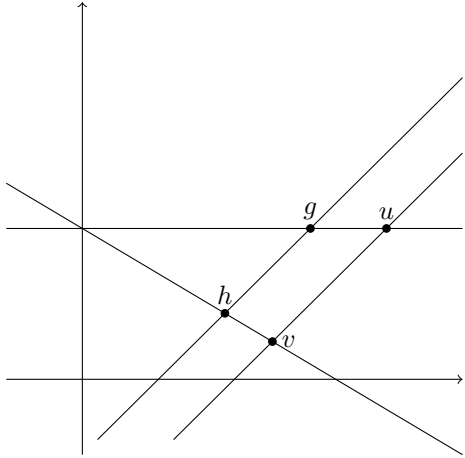
\begin{figure}[h!]
     \centering
     \begin{tikzpicture}
\draw [->] (0,-1) -- (0,5);
\draw [->] (-1,0) -- (5,0);

\draw (0.2,-0.8) -- (5,4);
\draw (1.2,-0.8) -- (5,3);
\draw (-1,2) -- (5,2);
\draw (-1,2.6) -- (5,-1);

\draw[fill] (3,2) circle [radius=0.05];
\draw[fill] (4,2) circle [radius=0.05];
\draw[fill] (1.875,0.875) circle [radius=0.05];
\draw[fill] (2.5,0.5) circle [radius=0.05];

\node[above] at (3,2) {$g$};
\node[above] at (4,2) {$u$};
\node[above] at (1.875,0.875) {$h$};
\node[right] at (2.5,0.5) {$v$};

\end{tikzpicture}
     \caption{An example of a non-degenerate quadrangle $(g,h,u,v) \in \mathcal Q(P)$}
     \label{fig:my_label}
 \end{figure}
 
 Because $|Q(P)| \leq E(P)$, we have the following corollary of Theorem \ref{thm:main}. In fact, the two quantities $|\mathcal Q(P)|$ and $E(P)$ are equal, modulo collinear energy quadruples.

\begin{corollary}\label{quadrangles} Let $P\subset \mathbb F^2$ such that no line contains more than $M$ points of $P$, and no vertical line contains more than $m$ points of $P$. Also, suppose in positive characteristic $p$ one has  $m|P|\leq p^2$. Then
$|\mathcal Q(P)| \leq m^{\frac{1}{2}}|P|^{\frac{5}{2}} +  M|P|^2 .$
\end{corollary}

We remark that if $P$ is an arithmetic or geometric progression cross itself, then it can be seen that $|\mathcal Q(P)|\gg |P|^{\frac{5}{2}}.$ In the special case when $P$ does not have more than $\sqrt{|P|}$ collinear points, Corollary \ref{quadrangles} yields $|\mathcal Q(P)|\ll |P|^{\frac{5}{2}+\frac{1}{4}}.$ Note that, over the reals and with the condition that no more than $\sqrt {|P|}$ are collinear, one has the bound $\ll |P|^3\log |P|$ for the number of quadruples $(g,h,u,v)$ such that the direction from $g$ to $h$ is equal to that from $u$ to $v$.
 See \cite[Theorem 5]{RuS}. After a projective transformation, the same bound holds for the number of quadruples $(g,h,u,v)$ such that the line connecting $g$ and $u$ has the same $y$-intercept as the line connecting $h$ and $v$. In view of the above lower bound, the events of the lines along pairs of opposite sides of the quadrangles in $\mathcal Q(P)$ meeting the lines $l_1$ and $l_2$ at the same point on the latter lines cannot be regarded as independent.

\section{Proofs}

 \subsection[Proof of Theorem 2.1]{Proof of Theorem \ref{thm:main}}
We adopt the convention that for any $x \in \emph{Aff}(\mathbb{F})$, $x = (x_1,x_2)$.
 \begin{proof}
 We prove the estimate for the quantity $E(A)$, the proof for $E^*(A)$ follows from the inequality $E^*(A)\leq E(A)$ from \cite{Sh}. Alternatively the forthcoming proof has an isomorphic version with $E^*(A)$ in place of $E(A)$.
 
Set
\begin{equation}
Q(A):=\{(g,h,u,v) \in A^4 :g^{-1} \circ h = u^{-1} \circ v\},
\label{e:quad}\end{equation}
so that $E(A) = |Q(A)|$. The first lemma gives a convenient decomposition of this set.

\begin{lemma} \label{lem:decomposition}
Let $A \subset G$ be finite. Then we have
\[
E(A)=\sum_{C} |\{(g,h,u,v) \in Q(A) : g_1v_1=C=h_1u_1\}|.
\] 
\end{lemma}

\begin{proof}
 Suppose that $g^{-1} \circ h = u^{-1} \circ v$. Then
 \[
 (g_1^{-1}, g_2g_1^{-1}) \circ (h_1,h_2) =  (u_1^{-1}, u_2u_1^{-1}) \circ (v_1,v_2).
 \]
 Calculating the first coordinates then gives
 \[
 h_1g_1^{-1}=v_1u_1^{-1}
 \]
 which rearranges to $g_1v_1=h_1u_1$.
\end{proof}

Let $Q_C$ denote the set of solutions corresponding to the value $C$ in the decomposition above, that is,
\[Q_C:= |\{(g,h,u,v) \in Q(A) : g_1v_1=C=h_1u_1\}|.
\]
Also define the set 
\[
\mathcal{C}_C : = \{(g,v) \in A \times A : g_1v_1=C\}.
\]
We bound $Q_C$ by the following lemma.

\begin{lemma} \label{lem:Cfixed}
 For any $C \in \mathbb{F}^*$, under the constraint that $m|A|\leq p^2$ in positive characteristic, one has
 \[
 Q_C \ll |\mathcal C_C|^{3/2} +M|\mathcal C_C|.
 \]
 \end{lemma}
 \begin{proof}[Proof of Lemma \ref{lem:Cfixed}]
 By considering the line equation $g^{-1} \circ h = u^{-1} \circ v$ and expanding out, we have the equations 
\begin{equation}\label{eneqs} g_1v_1 = h_1u_1 = C, \qquad \frac{h_2 -g_2}{g_1} = \frac{v_2 - u_2}{u_1}.\end{equation}
 The second equation can be written as 
\begin{equation} \label{ppinc} u_2g_1  - u_1g_2 - g_1v_2+ u_1h_2 = 0\,.\end{equation}
 We therefore map $\mathcal C_C\times \mathcal C_C$ to $\mathbb{P}^3\times \mathbb{P}^{3*}$ (where $\mathbb{P}^3$ is the three-dimensional projective space over $\mathbb{F}$ and $\mathbb{P}^{3*}$ is its dual) as
 $$
 (g,v)\to (g_1:g_2:g_1v_2:1),\qquad (u,h)\to (u_2:-u_1:-1:u_1h_2)\,.
 $$
The map is injective.  Indeed, suppose we have $(g_1,g_2,g_1v_2) = (g_1',g_2',g_1'v_2')$. Then we must have $(g_1,g_2) = (g_1',g_2')$, and since $g_1v_1 = g_1'v_1' = C$, we must also have $v_1 = v_1'$. The last coordinate then gives $g_1v_2 = g_1'v_2' \implies v_2 = v_2'$. 
 
Thus equation \eqref{ppinc}  can be viewed as a point plane incidence between the set of points $$P_C = \{ (g_1,g_2,g_1v_2) : (g,v) \in \mathcal{C}_C \}$$ and the set of planes $$\Pi_C = \{ \pi_{u,h} : (u,h) \in \mathcal{C}_C \},$$ where $\pi_{u,h}$ is the plane given by $u_2x - u_1y -z + u_1h_2 =0.$ Note that we have $|P_C| = |\Pi_C| = |\mathcal{C}_C|$. 

This calls for using the point-plane incidence bound from \cite{Ru}:
 \begin{theorem}[Rudnev]\label{thm:pointplane}
 Let $\mathbb{F}$ be a field, and let $P$ and $\Pi$ be finite sets of points and planes respectively in $\mathbb{P}^3$. Suppose that $|P| \leq |\Pi|$, and that $|P| \ll p^2$ if the characteristic $p\neq 2$ of $\mathbb{F}$ is positive. Let $k$ be the maximum number of collinear points in $P$. Then the number of incidences satisfies
 $$I(P,\Pi) \ll |\Pi||P|^{1/2} + k |\Pi|$$
 \end{theorem}
 Note that if $\mathbb{F}=\mathbb{F}_p$, one can dispense with the $p$-constraint in Theorem \ref{thm:pointplane}, replacing its estimate by the following asymptotic version (see, e.g. \cite[Lemma 6]{RudSh}):
 \begin{equation}\label{pp_ass}
 I(P,\Pi) -\frac{|\Pi||P|}{p} \ll |\Pi||P|^{1/2} + k |\Pi|\,.
 \end{equation}
 
 We claim that we can set $k=M$ in our application of Theorem \ref{thm:pointplane}. Indeed, suppose that a line contains more than $M$ points from $P_C$, each of the form $(g_1,g_2,g_1v_2)$. If this line is not of the form $\{(x_0,y_0,z) : z \in \mathbb F\}$ then project using the map $\pi(x,y,z)=(x,y,0)$ to get a line in the plane which contains more than $M$ points of the form $(g_1,g_2)$. This contradicts the assumption that $A$ contains no more than $M$ points on a line.

If the line is of the form $\{(x_0,y_0,z) : z \in \mathbb F\}$ and it contains more than $M$ points from $P_C$, it follows that there are more than $M$ points $(v_1,v_2)$ in $A$, all satisfying $v_1=C/x_0$.

To conclude the proof we apply Theorem \ref{thm:pointplane} to the point-plane arrangement $P_C$ and $\Pi_C$. We note the trivial (yet easily seen to be achievable) supremum bound \begin{equation}\label{triv}
|\mathcal C_C| \leq m|A|\,,
\end{equation}
which determines the constraint in positive characteristic, and then
have
 $$|Q_C| \leq I(P_C, \Pi_C) \ll |\mathcal{C}_C|^{3/2} + M |\mathcal{C}_C|\,,$$
 as required.
 \end{proof}

The claim of Theorem \ref{thm:main} now follows immediately. Apply Lemma \ref{lem:decomposition} and bound each summand by Lemma \ref{lem:Cfixed} to get
\[
E(A) \leq \sum_C \left( |\mathcal C_C|^{3/2} +M|\mathcal C_C| \right).
\]
Applying the  $L^\infty$ bound \eqref{triv} and the $L^1$ identity
$$
\sum_C|\mathcal{C}_C|=|A|^2\,
$$
completes the proof.
\end{proof}

 Lemma \ref{lem:Cfixed} cannot be improved without additional restrictions. Indeed, consider the case $A=[n] \times [n]$.
If we take $C=q$ for some prime $q \leq n$, we have $|\mathcal C_C|=2 n^2$. We also have $M=n$ (hence $M$ is a constant multiple of $|\mathcal C_C|^{1/2}$ - the regime in which the two terms in Lemma \ref{lem:Cfixed} are of the same order of magnitude). Lemma \ref{lem:Cfixed} tells us that $Q_C \ll n^3$. On the other hand, every solution to the additive energy equation
\[
b_1+b_2=b_3+b_4 ,\,\,\,\, b_i \in \{1,\dots, n\}
\]
gives a contribution to $Q_C$, and so $Q_C \gg n^3$.

In fact, if we instead take $C$ to be any element of the product set of $[n]$, then since $C$ has $\ll n^{\epsilon}$ representations as a product we have $|\mathcal C_C| \ll n^{2+\epsilon}$, and so the bound is near-optimal for all $C$ for this choice of $A$.

Moreover, estimate \eqref{triv} alone is also sharp if one simply takes $A=C\times D$, where $C$ is a geometric progression. 

\subsection[Proof of Theorem 2.3]{Proof of Theorem \ref{t1:Elekes}}	\begin{proof}
	We apply Theorems \ref{t:Elekes_new} and \ref{thm:main}. The constraints on $\alpha$ in the statement of Theorem \ref{t1:Elekes} have been chosen to ensure dominance of the first terms in the incidence bounds \eqref{f:Elekes_new_R}, \eqref{f:Elekes_new_Fp}. Let $L$ be the set of $k$ lines described in the statement of Theorems \ref{thm:main} and  \ref{t1:Elekes}. 
	
	For $\mathbb{F}=\mathbb{R}$ an application of Theorems \ref{t:Elekes_new} and \ref{thm:main} combined with the fact that $\alpha n \geq C' n^{1/2}$ gives
	$$
	\alpha k n \leq I(A \times A, L)  \ll n^{7/6} \left( m^{1/12} k^{3/4} + M^{1/6}k^{2/3} \right)\,,
	$$
	so either $m\gg \alpha^{12} n^{-2}k^{3}$ or $M\gg \alpha^{6} n^{-1}k^2$. It follows that either $\gg \alpha^{12} n^{-2}k^{3}$ lines, as elements of  $G$, lie in the same coset of $U$, so they are parallel, or $\gg \alpha^{6} n^{-1}k^2$ lines lie in the same coset of some torus, so they are concurrent.
	
	Similarly, for $\mathbb{F}=\mathbb{F}_p$, Theorems \ref{t:Elekes_new} and \ref{thm:main} combined (the condition $k \leq p$ ensures that $mk \leq p^2$) with the fact that $\alpha n \geq C' n^{1/2}$ and $p \geq \alpha^{-2} n$ give
	$$
	\alpha k n \ll n^{9/8} \left( m^{1/16} k^{13/16} + M^{1/8}k^{3/4} \right)\,,
	$$
	hence either $\gg \alpha^{16} n^{-2}k^3$ lines are parallel or $\gg \alpha^{8} n^{-1}k^2$ lines are concurrent.  
	
	To address the second claim of the first bullet point, parameterise the set of parallel lines by 
	 by $(\gamma, \beta)$, where $\gamma\neq 0$ is the fixed slope, and denote by $B \subseteq \mathbb{F}_p$ the set of all such $\beta$. By the assumption that we are in the case of the first bullet point, we have
	 $$
	 |B|\gg \alpha^{C} n^{-2}k^{3}\,,
	 $$
	 where $C=12$ or $C=16.$
	Hence, applying Cauchy-Schwarz twice, one has (with $r_{F(A)}(\beta)$ standing for the number of realisations of $\beta$ as a value of a function $F$, with several variables in $A$, and $A(x)$ being the indicator function of $A$):
	\begin{equation}\label{f:mixed_AS}\begin{aligned}
	\alpha n |B| & \le  \sum_{\beta \in B} \sum_{x \in A} A(\gamma x +\beta) = \sum_{\beta \in B} r_{A-\gamma A} (\beta) \le \sqrt{|B|} \left(\sum_{\beta \in B} r^2_{A-\gamma A} (\beta) \right)^{1/2} \\ & 
	\leq \sqrt{|B|} \sqrt{E^+(A,\gamma A)} \leq \sqrt{|B|} \sqrt{E^+(A)}\,,\end{aligned}
	\end{equation}
	proving the second claim of the first bullet point.

	Similarly in the case of $\gg\alpha^{C/2} n^{-1}k^2$ lines being concurrent at some point $(x_0,y_0)$, we parameterise them by their slopes, once again denoted as $\beta = \frac{y-y_0}{x-x_0}$ in a set $B$, with $|B|\gg\alpha^{C/2} n^{-1}k^2$.
	
	As in \eqref{f:mixed_AS} above, we have 
	$$
	\alpha n |B|  \le \sum_{\beta \in B} r_{\frac{A-y_0}{A-x_0}} (\beta)
	\le \sqrt{|B|} ({E^\times} (A-x_0){ E^\times} (A-y_0))^{\frac{1}{4}}\,,
	$$
	proving the second claim of the second bullet point.

	\end{proof}

\subsection[Proof of Corollaries 2.4, 2.5, 2.6]{Proof of Corollaries \ref{thm:lineintersection}, \ref{cor:cartesianintersection}, \ref{quadrangles}}
In this section we first apply Theorem \ref{thm:main} to prove Corollary \ref{thm:lineintersection}, whence we deduce Corollary \ref{cor:cartesianintersection}. At the end we prove Corollary \ref{quadrangles}.

\begin{proof}[Proof of Corollary \ref{thm:lineintersection}]
Let $l_1$, $l_2$ be distinct arbitrary lines, and let $P$ be a set of points in the plane with no more than $|P|^{1-\delta_1}$ points of $P$ on a line through the intersection of $l_1$ and $l_2$, and no more than $|P|^{1-\delta_2}$ points on any other line. By assumption, no more than $|P|^{1- \delta_1}$ points lie on $l_1$ or $l_2$, so we may remove such points and be left with a positive proportion of $P$. We apply a projective transformation $\pi$ to send $l_1$ to the $y$-axis $l_y$, and $l_2$ to the line at infinity $l_{\infty}$. Defining $P' := \pi(P)$, we have 
$$|L(P) \cap l_1| = |L(P') \cap l_y|, \quad |L(P) \cap l_2| = |L(P') \cap l_{\infty}|$$
which follows since projective transformations preserve incidence structure. Let $T$ denote the set $L(P') \cap l_y$ and $S$ denote the set $L(P') \cap l_{\infty}$. Note that lines through the intersection of $l_1$ and $l_2$ have been mapped to vertical lines. We claim that we have 
$$I(P', L(P')) \leq I(S \times T, P')$$
where on the right side $P'$ is being considered as a subset of the affine plane $\emph{Aff}(\mathbb{R})$ corresponding to a set of lines. The claim follows by taking a point $(p_1,p_2) \in P'$ lying on some line $y = sx + t$ coming from $L(P')$. By definition we must have $(s,t) \in S \times T$, and therefore this incidence may be viewed as the point $(s,t) \in S \times T$ lying on the line $y = -p_1 x + p_2$, proving the claim.

We bound these incidences using Theorem \ref{t:Elekes_new}, applied to the point set $S \times T$ and the line set $P'$ (strictly we are applying it to the line set given by $(-p_1,p_2)$ such that $(p_1,p_2) \in P'$, but this changes nothing with respect to the collinearity conditions). The lines in $P'$ are non-vertical and non-horizontal since we removed any points lying on $l_1$ or $l_2$ in $P$, which implies that $P'$ has only affine points lying off the $y$ axis. Furthermore, $P'$ has no more than $|P|^{1-\delta_1}$ points on any vertical line, and no more than $|P|^{1-\delta_2}$ points on any other line. We may bound the energy $E(P')$ using Theorem \ref{thm:main}. Plugging in $m \ll |P'|^{1-\delta_1}$, $M \ll |P'|^{1- \delta_2}$, we have $E(P') \ll |P'|^{3 - \frac{\delta_1}{2}} + |P'|^{3 - \delta_2}$, implying that 
\begin{align*}
    I(S \times T, P') &\ll |T|^{1/2} |S|^{2/3} E(P')^{1/6} |P'|^{1/3} + |T|^{1/2}|P'| \\
    & \ll |T|^{1/2}|S|^{2/3}\left(|P'|^{3 - \frac{\delta_1}{2}} + |P'|^{3 - \delta_2}\right)^{1/6}|P'|^{1/3} + |T|^{1/2}|P'| \\
    & \ll |T|^{1/2}|S|^{2/3}|P'|^{\frac{5}{6} - \frac{\delta_1}{12}} + |T|^{1/2}|S|^{2/3}|P'|^{\frac{5}{6} - \frac{\delta_2}{6}} + |T|^{1/2}|P'|.
\end{align*}
Since each line in $L(P')$ intersects $P'$ in at least two places and by Beck's theorem \cite{Beck} we have $|L(P')| \gg |P'|^2$, we may bound below by
$$I(S \times T, P') \geq I(P', L(P')) \geq |L(P')| \gg |P'|^2.$$
Putting this all together, we have
$$|P'|^2 \ll |T|^{1/2}|S|^{2/3}|P'|^{\frac{5}{6} - \frac{\delta_1}{12}} + |T|^{1/2}|S|^{2/3}|P'|^{\frac{5}{6} - \frac{\delta_2}{6}} + |T|^{1/2}|P'|.$$
If the third term dominates we have a stronger result, so we may assume one of the first terms dominates, giving
$$\max\{ |L(P) \cap l_1|, |L(P) \cap l_2|\} = \max\{ |S|,|T|\} \gg \min\{|P'|^{1 + \frac{\delta_1}{14}},|P'|^{1 + \frac{\delta_2}{7}} \}.$$
This gives the result since $|P| = |P'|$.
\end{proof}

To summarise the proof above, let $P'\subseteq P$ be the set of Beck points of $P$, each incident to $\gg|P|$ lines in $L(P)$. Then $P'$ can be viewed as a set of $\gg |P|$-rich affine transformations in the grid $S\times T$, and the claim of Corollary  \ref{thm:lineintersection} follows from the quantitative Elekes theorem, Theorem \ref{t1:Elekes}. This means that the qualitative answer to the question was, in fact, implied in the original Elekes' Theorem \ref{t:Elekes}.

\begin{proof}[Proof of Corollary \ref{cor:cartesianintersection}] 
Let $A \subseteq \mathbb{R}$ be a finite set, and $l$ an affine line not of the form $y = -x + k$ for some $k$, or the line $y = x$. We shall prove that $|L(A \times A) \cap l| \gg |A|^{2 + \frac{1}{14}}$. 

Let $\gamma$ denote the transformation given by reflection in the line $y=x$. The Cartesian product $A \times A$ is sent to itself under this reflection. 
Furthermore, by the restriction placed on the line $l$, we must have that $l':= \gamma(l) \neq l$. 

Since $A \times A$ is symmetric under $\gamma$, it follows that $L(A \times A)$ is also symmetric under $\gamma$, and moreover that the number of intersections of $L(A \times A)$ with $l$ is precisely the number of intersections of $L(A \times A)$ with $l'$, i.e. $|L(A \times A) \cap l| = |L(A \times A) \cap l'|$. An application of Corollary \ref{thm:lineintersection} with $\delta_1 = \delta_2 = 1/2$ then gives
$$|L(A \times A) \cap l| = \max\{ |L(A \times A) \cap l|,|L(A \times A) \cap l'|\} \gg |A|^{2 + \frac{1}{14}}.$$
\end{proof}

Note that one may consider a more general version of Corollary \ref{cor:cartesianintersection}, applying to a point set in $\mathbb{R}^2$ with a symmetry.

\begin{proof}[Proof of Corollary \ref{quadrangles}]
Let us show that if a quadrangle $(g,h,u,v)$ is in $\mathcal Q(P)$, then it also satisfies the energy equation, that is $(g,h,u,v)\in Q(P)$, defined by \eqref{e:quad}. Moreover, the sets $Q(P),\mathcal Q(P)$ are equal up to the case of degenerate members $(g,h,u,v)\in Q(P)$, where the points $g,h,u,v$ are collinear, lying as elements of $\emph{Aff}(\mathbb{F})$ in the same coset of some torus $T$ (or $U$). In the sequel we will say {\em torus} and just use the notation $T$ for either case, unless there is something special as to $U$.

For the first implication, if 
$$ g^{-1}\circ h = u^{-1}\circ v = t $$\ then also
$$g\circ u^{-1}= h\circ v^{-1}=t'\,.
$$ 
Then $t,t'$, unless they are the identity, which corresponds to a trivial count, lie respectively in some unique tori $T,T'$. If say $t\in T$, this means $g,h$ lie in some left coset of $T$ and $u,v$ in some other left coset of $T$. This means the lines connecting the corresponding the sides $gh$ and $uv$ of the quadrangle $g,h,v,u$ are parallel. Similarly $g,u$ and $h,v$ lie in two distinct right cosets of $T'$, and hence the sides $gu$ and $hv$ of the quadrangle $g,h,v,u$ have the same $y$-intercept.

The converse is slightly less obvious. Suppose, $(g,h,u,v)\in \mathcal Q(P)$. We therefore have two tori $T,T'$ defined respectively  by the equal slope of the lines $gh$ and $uv$, and the equal $y$-intercept of the lines $gu$ and $hv$. Hence, there are two left cosets of $T$, with representatives $c_1,c_2$, and elements $n_1,n_2,\nu_1,\nu_2\in T$, as well as two right cosets of $T'$, with representatives $s_1,s_2$, and elements $m_1,m_2,\mu_1,\mu_2\in T'$, so that
$$
g=c_1n_1=m_1s_1,\qquad h = c_1n_2=\mu_1s_2,\qquad u=c_2\nu_1=m_2s_1,\qquad v=c_2\nu_2=\mu_2s_2\,.
$$
This implies, using the Abelian notation within the commutative $T,T'$, that  in particular
$$
(v^{-1}\circ u) + (g^{-1}\circ h) = n_2-n_1+\nu_1-\nu_2 = s^{-1} \circ (m_2-m_1+\mu_1-\mu_2)\circ s^{-1}\,,
$$
where $s=s_1$ or $s= s_2$. It follows that $m_2-m_1+\mu_1-\mu_2=n_2-n_1+\nu_1-\nu_2$ equals identity. Indeed, suppose not.  If say the left-hand side of the latter equation is in the normal subgroup $U$, then so is the right-hand side. This means that the quadrangle $(g,h,u,v)$ has four vertical sides, namely the points $g,h,u,v$ lie on the same vertical line. The rest of the tori are each other's conjugates, meeting at the identity only. Then one has $s_1=s_2$, once again making the points $g,h,u,v$ collinear.

\end{proof}

\section[Appendix]{Appendix: Bisector Energy, Beck's Theorem and Point-Plane Incidences}

The approach taken to bounding affine energy in this paper was inspired by a series of works bounding the \textit{bisector energy} of a point set. Given a finite point set $P \subset \mathbb F$, its bisector energy is defined to be the number of solutions to
\begin{equation} \label{benergy}
B(p,q)=B(s,t),\,\,\,\,\,\,\,\, p,q,s,t \in P,
\end{equation}
 where $B(p,q)$ denotes the perpendicular bisector of $p$ and $q$. 
 
 Let $\mathcal B(P)$ denote the bisector energy of a set $P$. The situation is rather similar to what we face when trying to bound the affine group energy. A trivial bound is $\mathcal{B}(P) \leq |P|^3$, and taking the point set to be equidistributed on a line or circle realises $\mathcal B(P) \gg |P|^3$. So, as with our situation, in order to make progress it is necessary to impose some conditions on the original set, ruling out highly collinear or co-circular situations.
 
Let $M$ denote the maximum number of points from $P$ on a single line or circle. The following bound was proven for $P \subset \mathbb R^2$ in \cite{LP}:
\begin{equation} \label{bisect}
    \mathcal B(P) \ll |P|^{5/2} + M|P|.
\end{equation}
The proof of \eqref{bisect} used Beck's Theorem \cite{Beck}, and since this is not known to hold in the finite field setting, \eqref{bisect} does not immediately translate to $\mathbb F_p^2$. However, in \cite{MRS} (see also \cite{MPPRS} that superseded it) , it was shown that Beck's Theorem was not necessary, and a more direct approach was given, making use of Theorem \ref{thm:pointplane}. A similar situation occurred in drafting this paper: an earlier draft used an application of Beck's Theorem to prove Lemma \ref{lem:Cfixed}, based on the approach of \cite{LP}, before we discovered that the proof can be simplified and generalised via a fairly straightforward application of Theorem \ref{thm:pointplane}.

These two instances have highlighted a connection between Beck's Theorem and Theorem \ref{thm:pointplane}. To illustrate this connection, we conclude the paper with a new, short proof of Theorem \ref{thm:pointplane} in the Euclidean setting which uses only Beck's Theorem as an input. This is a considerable simplification on the original proof from \cite{Ru}, which used deeper techniques and built on the incidence theory developed by Guth and Katz \cite{GK}. 

\begin{proof}[Proof of Theorem \ref{thm:pointplane} over $\mathbb R^3$]
 
Recall that the aim is to prove that
\begin{equation} \label{incegoal}
I(P,\Pi) \ll |\Pi||P|^{1/2} + |P||\Pi|^{1/2}  + k|\Pi|
\end{equation}
holds for any set $P$ and $\Pi$ of points and planes in $\mathbb R^3$ respectively, such that no more than $k$ points from $P$ are contained in a single line. By Cauchy-Schwarz,
\begin{align*}
I(P, \Pi) &= \sum_{ \pi \in \Pi} |\{ p \in P : p \in  \pi\}| \leq |\Pi|^{1/2} \left(\sum_{ \pi \in \Pi} |\{ p \in P : p \in  \pi\}|^2\right)^{1/2}
\\ & = |\Pi|^{1/2} \left(I(P, \Pi) + \sum_{\pi \in \Pi}|\{(p_1,p_2) \in (P \cap \pi)^2  : p_1 \neq p_2 \}|\right)^{1/2}.
 \end{align*}
 If the first term in the bracket dominates then we have $I(P, \Pi) \ll |\Pi|$, which is much stronger than required. Therefore we can disregard it, and assume that
 \begin{equation} \label{halfway}
 I(P, \Pi) \ll |\Pi|^{1/2} \left( \sum_{\pi \in \Pi}|\{(p_1,p_2) \in (P \cap \pi)^2  : p_1 \neq p_2 \}|\right)^{1/2}.
 \end{equation}
 In each plane $\pi$, Beck's Theorem tells us that there are two possibilities
 \begin{itemize}
     \item[(i)]  $\gg |\pi \cap P|^2$ of the pairs of distinct points from $\pi \cap P$ are contained on a single line,
     \item[(ii)] $\gg|\pi \cap P|^2$ of the pairs of distinct points from $\pi \cap P$ are contained on lines which support less than $C$ points from $ \pi \cap P$, for some absolute constant $C$.
 \end{itemize}
Let $\Pi_1$ be the set of all planes of type (i), and $\Pi_2$ those of type (ii).  We seek to bound
\begin{align} \label{splitsum}
\sum_{\pi \in \Pi}|\{(p_1,p_2) \in (P \cap \pi)^2  : p_1 \neq p_2 \}| &\leq \sum_{\pi \in \Pi_1}|\{(p_1,p_2) \in (P \cap \pi)^2  : p_1 \neq p_2 \}| \nonumber
\\& + \sum_{\pi \in \Pi_2}|\{(p_1,p_2) \in (P \cap \pi)^2  : p_1 \neq p_2 \}|.
\end{align}
Since there are no more than $k$ points from $P$ on any line, the first term above contributes at most $k^2|\Pi|$. For the second term, we use Cauchy-Schwarz to obtain
\begin{align*}
\sum_{\pi \in \Pi_2}|\{(p_1,p_2) \in (P \cap \pi)^2  : p_1 \neq p_2 \}| &\leq C^2 \sum_{ l \in L(P) } | \{\pi \in \Pi : l \in \pi \}|
\\& \ll |L(P)|^{1/2} \left( \sum_{l \in L(P)} \sum_{\pi_1, \pi_2 \in \Pi :l \in \pi_1 \cap \pi_2} 1      \right)^{1/2}
\\ & \ll |P|\left( |\Pi|^2 + \sum_{ l \in L(P) } | \{\pi \in \Pi : l \in \pi \}|\right )^{1/2}
\end{align*}
If the second term is dominant then we get
\[
\sum_{\pi \in \Pi_2}|\{(p_1,p_2) \in (P \cap \pi)^2  : p_1 \neq p_2 \}| \ll |P|^2.
\]
Otherwise the first term is dominant and we get $|P||\Pi|$ on the right hand side. Therefore,
\[
\sum_{\pi \in \Pi_2}|\{(p_1,p_2) \in (P \cap \pi)^2  : p_1 \neq p_2 \}| \ll |P|(|P|+|\Pi|).
\]
Combining this with \eqref{halfway} and \eqref{splitsum}, we finally conclude that
\[
I(P, \Pi) \ll k|\Pi| +|P||\Pi|^{1/2} + |P|^{1/2}|\Pi|,
\]
as required.
\end{proof}

One may compare the proof of \eqref{incegoal} above with the work of Elekes and Toth \cite{ET}, who also proved a point-plane incidence bound in $\mathbb R^3$ with a slightly different notion of non-degeneracy. The bound in \cite{ET} is better in the asymmetric case. Rather than using Beck's Theorem \cite{Beck}, the proof in \cite{ET} uses the full strength of the Szemer\'{e}di-Trotter Theorem \cite{Sze-Tro}.

\bigskip

\noindent {\bf Acknowledgements}. The project begun while the second and third listed authors visited The University of Georgia. The first listed author is supported by the NSF Award 1723016 and gratefully acknowledges the support from the RTG in Algebraic Geometry, Algebra, and Number Theory at the University of Georgia, and from the NSF RTG grant DMS-1344994. The second and fourth listed authors were partially supported by the Austrian Science Fund FWF Project P 30405-N32. The third listed author was partially supported by the Leverhulme grant RPG-2017-371. We thank Brendan Murphy and Adam Sheffer for helpful discussions and also Ilya Shkredov for his feedback on an earlier draft of the paper. Finally, we thank an anonymous referee for their comments and corrections.

\phantomsection

\bigskip

\noindent Giorgis Petridis\\
Department of Mathematics, University of Georgia, Athens, GA 30602, USA.\\
giorgis@cantab.net

\medskip
\noindent Oliver Roche-Newton\\
Johann Radon Institute for Computational and Applied Mathematics \\ 69 Altenberger Stra{\ss}e,
Linz 4040, Austria.\\
o.rochenewton@gmail.com

\medskip
\noindent Misha Rudnev\\
School of Mathematics, Fry Building, Bristol BS8 1TH, UK.\\
misharudnev@gmail.com

\medskip
\noindent Audie Warren\\
Johann Radon Institute for Computational and Applied Mathematics \\
69 Altenberger Stra{\ss}e,
Linz 4040, Austria.\\
audie.warren@oeaw.ac.at


\addcontentsline{toc}{section}{References}


\begin{thebibliography}{19}

\bibitem{Beck}
Beck, J. 
``On the lattice property of the plane and some problems of Dirac, Motzkin and Erd\H{o}s in combinatorial geometry."
\textit{Combinatorica} 3, no. 3-4 (1983):281--297.

\bibitem{Elekes1}
Elekes, G.
``On linear combinatorics. I. Concurrency: an algebraic approach.''
\textit{Combinatorica} 17, no. 4 (1997):447--458.

\bibitem{ET}
Elekes, G. and T\'{o}th, C. D. 
``Incidences of not-too-degenerate hyperplanes.''
\textit{SCG '05: Proceedings of the twenty-first annual symposium on Computational geometry} (2005):16--21.


\bibitem{Gill}
Gill, N.
``Quasirandom group actions.''
\textit{Forum of Mathematics, Sigma} 4 (2016), E24: 35 pp.

\bibitem{GK} 
Guth, L. and Katz, N. H. 
``On the Erd\H{o}s distinct distance problem in the plane.'' 
\textit{Ann. of Math. (2)} 181, no. 1 (2015):155--190.

\bibitem{HaH}  
Helfgott, H. A.  
``Growth in groups: ideas and perspectives.'' 
\textit{Bull. Amer. Math. Soc.} 52, no. 3 (2015): 357--413.

\bibitem{LP}  
Lund, B. and Petridis, G.
``Bisectors and pinned distances.'' Accepted in 
\textit{Discrete Comput. Geom.}, arXiv:1810.00765 (2018).

\bibitem{LSdZ}  
Lund, B. and Sheffer, A. and de Zeeuw, F. 
``Bisector energy and few distinct distances'' 
\textit{Discrete Comput. Geom.} 56, no. 2 (2016):337--356.


\bibitem{Brendan_rich}
Murphy, B.
``Upper and lower bounds for rich lines in grids.'' 
Accepted in \textit{Amer. J. Math.}, arXiv:1709.10438 (2017).

\bibitem{MPPRS}  
Murphy, B. and Petridis, G. and Pham, T. and Rudnev, M. and Stevens, S. 
``On the pinned distances problem over finite fields.''
arXiv:2003.00510 (2020).

\bibitem{MRS}  
Murphy, B. and Rudnev, M. and Stevens, S. 
``Bisector energy and pinned distances in positive characteristic.''
arXiv:1908.04618 (2019).


\bibitem{RNW} 
Roche-Newton, O. and Warren, A.
``New expander bounds from affine group energy.'' Accepted in 
\textit{Discrete Comput. Geom.}, arXiv:1905.03701 (2019).

\bibitem{Ru}  
Rudnev, M.
``On the number of incidences between points and planes in three dimensions.'' 
\textit{Combinatorica} 38 (2018): 219--254.

\bibitem{RuS}  
Rudnev, M. and Selig, J. M. 
``On the use of the Klein quadric for geometric incidence problems in two dimensions.'' 
\textit{SIAM J. Discrete Math.} 30 , no. 2 (2016): 934--954.

\bibitem{RuSh}  
Rudnev, M. and Shkredov, I. D. 
``On growth rate in $SL_2(\mathbb F_p)$, the affine group and sum-product type implications.''
arXiv:1812.01671 (2018).

\bibitem{RudSh} 
Rudnev, M. and Shkredov, I. D. 
``On the restriction problem for discrete paraboloid in lower dimension.'' 
\textit{Adv. Math.} 339 (2018):657--671.

\bibitem{Sh}  
Shkredov, I. D.
``Modular hyperbolas and bilinear forms of Kloosterman sums.'' arXiv:1907.03357 (2019).

\bibitem{Sze-Tro}
Szemer{\'e}di, E. and Trotter, W. T.
``Extremal problems in discrete geometry''
\textit{Combinatorica} 3, no. 3--4 (1983):381--392.


\bibitem{Sz}  
Sz\H{o}nyi, T. 
``Around R\'edei's theorem.''
\textit{Discrete Math.} 208/209 (1999):557--575 (1999).
\end{thebibliography}
\end{document}